\documentclass{amsart}
\usepackage{amsmath,amssymb,latexsym,color}
\usepackage{hyperref}
\newtheorem{cor}{Corollary}%[section]

\newtheorem{lemma}{Lemma}

\newtheorem{prop}{Proposition}

\newtheorem*{namedthm*}{\namedthmname}
\newenvironment{namedthm}[1]
  {\newcommand\namedthmname{#1}\begin{namedthm*}}
  {\end{namedthm*}}
\newcommand{\ov}[1]{\overline{#1}}

\begin{document}
 \author[Fumagalli]{Francesco Fumagalli}
    \address{Dipartimento di Matematica  e Informatica ``U.\,Dini" \\
   Universit\`a di Firenze\\
   Viale Morgagni 67A, \ I-50134 Firenze, Italy}
   
   \email{francesco.fumagalli\,@\,unifi.it}
   \author[Leinen]{Felix Leinen}
   \address{Institut f\"ur Mathematik \\
   Johannes Gutenberg-Universit\"at\\
   D$-$55099 \,Mainz, Germany}
   \email{Leinen\,@\,uni-mainz.de}
   
   \author[Puglisi]{Orazio Puglisi}
   \address{Dipartimento di Matematica  e Informatica ``U.\,Dini"\\
   Universit\`a di Firenze\\ 
   Viale Morgagni 67A, \ I-50134 Firenze, Italy}
   \email{orazio.puglisi\,@\,unifi.it}
   
\title{Bounding the Fitting height\\ in terms of the exponent}
% Words, permutations, and the nonsolvable length of a
% finite group (Shalev)
\date{\today}
\dedicatory{In memory of our friend Carlo Casolo}
\begin{abstract}
Every finite solvable group $G$ has a normal series with nilpotent factors.
The smallest possible number of factors in such a series is called the Fitting height $h(G)$.
In the present paper, we derive an upper  bound for  $h(G)$ in terms of the exponent of $G$.
Our bound constitutes a considerable improvement of an earlier bound obtained in \cite{sha}.
\vskip1ex
\noindent
MSC2020$\colon$ 20D10, 20F14, 20F16.\\ 
Keywords$\colon$ Fitting height, exponent.
   
\end{abstract}
\maketitle

%%%%%%%%%%%%%%%%%   Section 1     %%%%%%%%%%%%%%%%%%%%%%%%%%%%%

\pagestyle{myheadings}
\markright{{\large\sc an upper bound for the Fitting height}}
\markleft{{\large\sc fumagalli $\mid$ leinen $\mid$ puglisi}}

\section{Introduction}

Let $F(G)$ denote the Fitting subgroup of the group $G$. 
The \emph{Fitting series} in $G$ is defined via  
$$F_0(G)=1\text{ \ and}\quad F_i(G)/F_{i-1}(G)=F(G/F_{i-1}(G))\quad\text{for every }i\ge 1\,.$$
When  $G$ is finite, this series reaches $G$ if and only if $G$ is solvable. 
The number of 
nontrivial members of the Fitting series is  called the
\emph{Fitting height} $h(G)$ of $G$.

%In the present note we will consider only finite solvable groups. 
The Fitting height seems to have a strong influence on the structure of a finite solvable group.
Quite a few investigations have been devoted to finding relations involving the Fitting height
and other invariants.  The interested reader may get an idea of such kind of problems
by consulting the survey \cite{tur} and its bibliography.
 
In this short note we will focus on the interplay between the Fitting height and the exponent.
% of finite solvable groups.
A.\,Shalev proved in \cite[Lemma 2.5]{sha} that
$$
h(G)< \prod_{i=1}^k(2e_i+1).
$$
whenever $G$ is a finite solvable group with exponent
$\exp(G)=p_1^{e_1}\ldots p_k^{e_k}$,
where the $p_i$ are pairwise distinct primes.
Let $\Omega(m)$ denote  the number of prime divisors of the natural number $m$, 
counted with multiplicities. 
Shalev's result implies the exponential bound
$h(G)< 3^{\Omega(\exp(G))}$.  

Let $\Omega_1(m)$ denote the number of odd prime divisors of $m$, which are Fermat primes 
(counted with multiplicities).
%, and let $\Omega_0(m)=\Omega(m)-\Omega_1(m)$.
We shall improve Shalev's inequality to a linear bound as follows.

\begin{namedthm}{Theorem 1}\label{thm}
Consider a finite solvable group $G$. Then  
$$
h(G)~\leq~ \Omega(\exp(G))+ \Omega_1(\exp(G))~\leq~ 2\,\Omega(\exp(G)).
$$
\end{namedthm}

A slightly nicer bound holds for groups of odd order.

\begin{namedthm}{Theorem 2}\label{thm2}
Consider a finite group $G$ of odd order. Then  
$$
h(G)\leq \Omega(\exp(G)).
$$
\end{namedthm}

The bound given by Theorem 1 is tight, 
whenever the order of $G$ is not divisible by any Fermat prime. 
To see this, let $P_i$ be an elementary abelian $p_i$-group, for distinct non-Fermat primes
$p_1, \dots , p_n$.
The iterated wreath product 
$$((...(P_1\wr P_2)\wr \ldots\,)\wr P_{n-1})\wr P_n$$ 
of the groups $P_i$ in their regular representations has Fitting height $n$ and exponent
$m=p_1\cdots p_n$; therefore $n=\Omega(m)$. 
This kind of example also shows, that the bound given in Theorem 2 is tight. 
It remains uncertain, whether equality can hold in Theorem 1 in the presence of Fermat primes.

It is worth remarking, that other common invariants of finite solvable groups cannot be
bounded by a function of the exponent. For example, as pointed out in \cite[p.\,267]{vauzel}, 
the derived length of the largest finite $m$-generated group of prime power exponent $p^n\ge4$
is at least $\lfloor \log_2 m\rfloor$. 
Therefore, the derived length of a $p$-group cannot 
be bounded by a function of its exponent.

%%%%%%%%%%%%%%%% Section 2 %%%%%%%%%%%%%%%%%%%%%%

\section{The result}

For any group $G$, let 
$$
\rho_0(G)=G,\quad
\rho_1(G)=\bigcap_{i\in \mathbb{N}}\gamma_i(G) \quad \text{and} \quad
\rho_n(G)=\rho_1(\rho_{n-1}(G))\quad \text{for all }~n\ge2.
$$
It is clear that $G$ has Fitting height $n$ if and only if $n$ is 
minimal with respect to $\rho_n(G)=1$.

Let $G$ be a finite group. A \emph{tower} of height $n$ in $G$ is a family
$\{P_n,\ldots,P_1\}$ of non-trivial subgroups of $G$ satisfying 
\begin{enumerate}
\item[(1)~] 
every ${P}_i$ is a $p_i$-group for some prime $p_i$,
where $p_{i+1}\neq p_i$ \ for \\ $i=n-1,\ldots,1$,
\item[(2)~] 
${P}_i$ normalizes ${P}_j$ whenever $i<j$, \ and
\item[(3)~] 
$[\ov P_i,P_{i-1}]=\ov P_i$ \ for \ $i=n,\ldots, 2$, \ where we 
use the notation\\   
$\ov P_i={P}_i/C_i$ with $C_n=1$ and $C_i=C_{P_i}(\ov P_{i+1})$  for $i=n-1,\ldots,1$.
%\item[(4)~] 
%$\ov P_i/\Phi(\ov P_i)$ \,is an irreducible \,$(P_{i-1}\ldots P_1)$-module for \,$i=n,\ldots,2$.
\end{enumerate}
Note, that property (3) implies the non-triviality of the groups $\ov P_i$ whenever $P_n\neq1$.
The above definition is a slightly weaker form of the concept of a tower as 
introduced in \cite{Turull}.
\begin{lemma}\label{towers_exist} 
Every non-trivial finite solvable group with Fitting height $n$ contains a tower of height $n$.
\end{lemma}

\begin{proof}
We shall call a family $\{P_n,\ldots,P_1\}$ of subgroups in $G$ a \emph{weak} 
tower of height $n$, if it satisfies properties (1) and (2) of a tower, and in addition 
\begin{enumerate}
\item[(3')~]
$\ov P_i\neq 1$ \,for all $i$.
\end{enumerate} 
By \cite[Lemma 1.4]{Turull}, it is sufficient to show the existence 
of a weak tower of height $n$ in $G$.

To this end, consider the Fitting series \ $1=F_0<F_1<\ldots<F_n=G$ \ in $G$. 
Here we have $F_n=G$,
because $G$ has Fitting height $n$. We shall proceed by induction over $n$ in order to produce
a weak tower $\{P_n,\ldots,P_1\}$ in $G$ such that 
\begin{itemize}
\item[($\star$)\quad]
$P_{n-i}F_{i}/F_i$ is a Sylow subgroup in 
$F_{i+1}/F_i$ for $0\le i\le n-1$.
\end{itemize}

When $n=1$, choose any non-trivial Sylow subgroup $P_1$ in $G$. 
Suppose then, that $n>1$. By induction, there exists a weak tower $T$ in $G/F_1$ satisfying 
($\star$).
From \cite[Lemma 1.6]{Turull}, this leads to a weak tower $\{P_{n-1},\ldots,P_1\}$ in
$G$ satisfying ($\star$), such that $T$ consists of the groups \,$P_iF_1/F_1$ \,$(n-1\ge i\ge 1)$.
For some prime $p_n\neq p_{n-1}$ the group $F_1$ contains a Sylow $p_n$-subgroup $P_n$
satisfying $[P_n,P_{n-1}]\neq 1$,
because otherwise 
the subgroup $P_{n-1}F_1$ would be a nilpotent normal subgroup in $F_2$
and thus be contained in $F_1$, a contradiction. 
Now  $\{P_n,\ldots,P_1\}$ is a weak tower in $G$.
\end{proof}

\begin{lemma}\label{quotient} 
Let $G$ be a non-trivial finite solvable group with Fitting height $n$.\linebreak
Suppose, that $G={P}_{n}\cdots{P}_{1}$ for some tower $\{P_n,\ldots,P_1\}$ in $G$ of height $n$. 
If $N\trianglelefteq G$ and $P_nN/N\neq 1$, then
$\{P_nN/N, \ldots, P_1N/N\}$ is a tower in $G/N$.
\end{lemma}

\begin{proof}
Let $Q_i=P_iN/N$ and $\ov Q_i=Q_i/D_i$ for all $i$, where $D_n=1$ and 
$D_i=C_{Q_i}(\ov Q_{i+1})$ for $i=n-1,\ldots,1$.
Obviously, the family $\{Q_n,\ldots,Q_1\}$ of subgroups of $G/N$ inherits 
properties (1) and (2) of a tower from the given tower in $G$. 
Since $Q_n\neq 1$ by hypothesis, the non-triviality of the groups $Q_i$ will be a consequence of 
property (3). Therefore it just remains to establish (3) for the family $\{Q_n,\ldots,Q_1\}$.

To this end, we will show first, that $C_{P_i}(\ov P_{i+1})N/N\leq D_i$ for $i=n-1,\ldots,1$.
This inclusion obviously holds for $i=n-1$.
We proceed by recursion and assume, that there exists $k\in\{n-1,\ldots,2\}$
such that the inclusion has already been shown for $i=n-1,\ldots,k$.
Then $\ov Q_k= Q_k/D_k$ is a homomorphic image of $P_kN/C_{P_k}(\ov P_{k+1})N$, hence of 
$\ov P_k=P_k/C_{P_k}(\ov P_{k+1})$. 
Since the involved homomorphisms are projections, and since all the involved groups are
normalized by $P_{k-1}$, we have that $\ov Q_k$ is $P_{k-1}$-isomorphic to a quotient of 
$\ov P_k$. 
Therefore, $C_{P_{k-1}}(\ov P_k)$ acts trivially on $\ov Q_k$, and 
we obtain $C_{P_{k-1}}(\ov P_k)N/N\leq D_{k-1}$.
This completes the recursion.

Now, for each $i\in\{n-1,\ldots,1\}$, the quotient  
$\ov Q_i$ is $P_{i-1}$-isomorphic to an image of $\ov P_i$.
Therefore property (3) of the given tower in $G$ implies that   
$[\ov Q_i, P_{i-1}] = \ov Q_i$, and it follows that $[\ov Q_i, Q_{i-1}] = \ov Q_i$
for each $i\in\{n-1,\ldots,1\}$.
\end{proof}

\begin{lemma}\label{contains}
Let $\{P_n,\ldots,P_1\}$ be a tower in a finite solvable group.
If $G={P}_n\cdots{P}_1$, then $\rho_{n-1}(G)={P}_n$.
\end{lemma}
\begin{proof}
The claim is true when $n<2$. So we suppose now, that $n\ge 2$ and argue by induction. 
Again, we consider the centralizer $C=C_{n-1}$. Let $X=P_{n-1}\cdots P_1$.
By Lemma \ref{quotient},
the images in  $G/C$  of the groups $P_n,\ldots,P_1$ form a tower of height $n$ in $G/C$.
Since $C_{P_{n-1}/C}(P_nC/C)=C_{n-1}/C=1$, the images in $G/C$ 
of the groups $P_{n-1},\ldots,P_1$
form a tower of height $n-1$, 
and $X/C$ is the product of these images. Therefore our inductive hypothesis yields 
${P}_{n-1}=\rho_{n-2}(X)C\le\rho_{n-2}(G)C$.  

We can use this relation in order to prove, that ${P}_n\leq \rho_i(G)$ for all $i<n$:
Beginning with ${P}_n=[{P}_n,{P}_{n-1}]\leq \gamma_2(G)$, a recursion shows that 
$$
P_n=[{P}_n,{P}_{n-1}]\le[\gamma_j(G),G]=\gamma_{j+1}(G)\quad\text{for all }j.
$$
It follows that  ${P}_n\leq \rho_1(G)$.

Arguing by induction, we suppose next, that ${P}_n\leq \rho_i(G)$ 
for some $i\leq n-2$. 
Then
\begin{eqnarray*}
P_n&=&[P_n,P_{n-1}]~\le~[P_n,\rho_{n-2}(G)C]\\
&=&[P_n,\rho_{n-2}(G)]
~\le~[\rho_i(G),\rho_{n-2}(G)]~\le~ \gamma_2(\rho_i(G))\,,
\end{eqnarray*}
and it follows as before, that $P_n\le\gamma_{j+1}(\rho_i(G))$ for all $j$. 
In particular ${P}_n\leq \rho_{i+1}(G)$. 
In the end, we obtain ${P}_n\leq \rho_{n-1}(G)$.

On the other hand, the subgroups $N_i={P}_n\cdots{P}_i$ are normal in $G$ 
and every $P_i$ is nilpotent. 
Therefore, $\rho_i(G)\leq{P}_n\cdots{P}_i$ for all $i$.
In particular, $\rho_{n-1}(G)\leq{P}_n$. 
\end{proof}

\begin{cor}\label{tower height}
If the finite solvable group $G$ is the product of subgroups, which form a tower of height $n$,
then $h(G)=n$.
\end{cor}

\begin{cor}\label{N_not_in_P_n}
Let $G$ be a finite solvable group with a tower
$\{P_n,\ldots,P_1\}$ such that $G={P}_n\cdots{P}_1$. 
If a normal subgroup $N$ of $G$ does not contain ${P}_n$, 
then $h(G/N)=h(G)=n$.
\end{cor}

\begin{proof}
If $h(G/N)<n$, then $\rho_{n-1}(G/N)=1$ and $P_n=\rho_{n-1}(G)\leq N$.
\end{proof}

When
$\mathcal{T}=\{{P}_n, \dots , {P}_1$\} is a tower in a finite solvable group $G$,
we define 
$$
m_p(\mathcal{T})~=~\big|\{i~\,|~\,\text{$P_i$ is a $p$-group}\}\big|\quad
\text{for each prime $p$.}
$$
Recall, that the $p$-length $\ell_p(G)$ of a finite solvable group $G$ is the number of
$p$-factors in a shortest normal series in $G$, whose factors are $p$-groups or $p'$-groups.

\begin{prop}\label{multiplicity}
Let the finite solvable group $G$ be the product of 
the subgroups in a tower $\mathcal{T}$. 
Then $\ell_p(G)=m_p(\mathcal{T})$ for all primes $p$.
\end{prop}

\begin{proof}
We shall proceed by induction on the Fitting height $n$ of $G$. 
Clearly, the claim holds for $n\le2$. 
So we assume now, that $n>2$ and that the claim is true for all groups of 
Fitting height $\le n-1$. 
Amongst the groups of Fitting height $\le n$, we proceed by induction over $|G|$.
We may thus assume, that $h(G)=n$ and that the claim holds for all groups $H$ of 
Fitting height $\le n$ satisfying $|H|<|G|$.
By hypothesis, $G$ is the product of subgroups 
${P}_n, \dots , {P}_1$ forming a tower $\mathcal T$.

Suppose, that there exists a non-trivial normal subgroup $N$ in $G$, 
which is properly contained in $P_n$. The factor group $G/N$ has Fitting height
$n$, and Lemma \ref{quotient} ensures that the subgroups $P_iN/N$ $(n\ge i\ge1)$ form a 
tower $\mathcal S$ in $G/N$. 
By minimality of $G$, we have $\ell_p(G/N)=m_p(\mathcal S)=m_p(\mathcal T)$
for all primes $p$. Now $O_{p_n}(G/N)=O_{p_n}(G)/N\ge P_n/N\neq 1$ implies 
$\ell_{p_n}(G)=\ell_{p_n}(G/N)=m_{p_n}(\mathcal T)$. Moreover, 
$\ell_{p}(G)=\ell_{p}(G/N)=m_{p}(\mathcal T)$ holds for all primes $p\neq p_n$ because 
$N$ is a $p_n$-group. 
We have thus shown, that it remains to treat the case, when 
$P_n$ is a minimal normal subgroup in $G$.

Let $V=P_n$ and $X={P}_{n-1}\cdots {P}_1$. Since $G=VX$ and since $V$ is abelian, 
we have $V\cap X\trianglelefteq G$. Hence $V\cap X=1$ or $V\cap X=V$.
However, the latter case cannot occur, because $V\le X$ and 
$P_{n-1}\trianglelefteq X$ would imply $V=[V,P_{n-1}]\le V\cap P_{n-1}=1$.
We have thus shown, that $G$ is the semidirect product of $V$ and $X$.

Consider the centralizer $C=C_{{P}_{n-1}}(V)$. Note that $C<P_{n-1}$ because of 
property (3) of the tower $\mathcal T$.
By Lemma \ref{quotient}, the subgroups $P_iC/C$ $(n-1\ge i\ge1)$ form a 
tower $\mathcal R$ in $X/C$. By minimality of $G$, we obtain
$m_p(\mathcal R)=\ell_p(X/C)$ for all primes $p$.
Since $C< P_{n-1}$, we also have $\ell_p(X/C)=\ell_p(X)$ for all primes $p$.
It follows, that $m_p(\mathcal T)=m_p(\mathcal R)=\ell_p(X)$ for all primes $p\neq p_n$.
It remains to treat the prime $p=p_n$.

Consider the subgroup $O_p(G)$ of $G$. In the case when $O_p(G)=V$, we have
$m_p(\mathcal T)=m_p(\mathcal R)+1=\ell_p(X/C)+1=\ell_p(X)+1=\ell_p(G)$.
Therefore, it remains to treat the case when $V<O_p(G)$.
Note that $VO_p(X)$ is a normal p-subgroup in $G$. 
Therefore $O_p(G)=VX\cap O_p(G)=V(X\cap O_p(G))\le VO_p(X)\le O_p(G)$, so that
equality holds. The non-trivial normal subgroup $Z(O_p(G))\cap V$ of $G$ must
coincide with the minimal normal subgroup $V$ of $G$. Thus, $Y=O_p(X)$ is
centralized by $V$, hence $Y\trianglelefteq G=VX$.

By Corollary \ref{N_not_in_P_n}, the group $G/Y=(P_nY/Y)\cdots(P_1Y/Y)$ 
has Fitting height $n$. It follows, that  $P_iY/Y\neq1$ for $i=n,\ldots,1$.
And Lemma \ref{quotient} ensures, that the subgroups $P_iY/Y$ form a tower $\mathcal U$ in $G/Y$.
By minimal choice of $G$, we have $m_p(\mathcal T)=m_p(\mathcal U)=\ell_p(G/Y)=\ell_p(G)$.
The proof of Proposition \ref{multiplicity} is complete.
\end{proof}

\begin{prop}\label{lengh=exp}
Let $G$ be a finite solvable group.
For each prime $p$, let  $p^{e_p}$ be the exponent of the Sylow $p$-subgroups of $G$
and let $\ell_p$ denote the $p$-length of $G$. 
Then we have 
\begin{enumerate}
\item~ $\ell_p\leq 2e_p$, whenever $p$ is an odd Fermat prime,
\item~ $\ell_p\leq e_p$, \ whenever $p=2$ or $p$ is odd and not a Fermat prime.
\end{enumerate}
\end{prop}

\begin{proof}
By \cite[Theorem A]{hh} we have $e_p\geq \lfloor (l_p+1)/2\rfloor$, whenever $p$ is an 
odd Fermat prime. It follows that $2e_p\ge\ell_p$.
In all other cases, $e_p\ge\ell_p$ follows from \cite[Theorem A]{hh} for odd $p$ and from
\cite{brj} for $p=2$.
\end{proof}

We can now relate the exponent of a finite solvable group to its Fitting height.
\vskip1ex

%\begin{namedthm}{Theorem.}\label{exponent}
%Let $G$ be a finite solvable group of Fitting height $n$. 
%Then  $\mathrm{exp}(G)\geq 2^{n/2}$.
%\end{namedthm}

\emph{Proof  of Theorem 1.} \
From Lemma \ref{towers_exist} 
it is enough to prove the claim, when $G$ is the product of subgroups 
forming a tower $\mathcal T$.
The exponent of $G$ is the product of  the exponents of its Sylow subgroups.
If $p_1, p_2, \ldots, p_k$ are the primes dividing $|G|$, where $p_i$ is Fermat for $i=r+1, \dots k$,  then we write
$p_i^{e_i}$ for the exponent of the Sylow $p_i$-subgroups of $G$. 
Propositions \ref{multiplicity} and \ref{lengh=exp} directly imply 
\newpage
\begin{eqnarray*}
h(G)&=& \sum_{i=1}^km_{p_i}(G)
~=~\sum_{i=1}^k\ell_{p_i}(G)\\
&\leq& \sum_{i=1}^re_{i}(G)+2\!\sum_{i=r+1}^ke_{i}(G)
~=~\Omega(\exp(G))+\Omega_1(\exp(G))\,.
\end{eqnarray*}
\phantom{.}\hfill$\Box$
\vskip1ex

\emph{Proof  of Theorem 2.} A group of odd order is solvable by Feit-Thompson theorem.  Again it is enough to prove the claim when $G=P_n\cdots P_1$ 
where $\{P_n,\dots , P_1\}$ is a tower. 
Since $2$ does not divide the order of $G$, 
it is a consequence of \cite[Theorem 2.1.1]{hh} and part (ii) of its Corollary, 
that the inequality $\ell_p\leq e_p$ holds for all prime divisors $p$ of $\left|G\right|$. 
We thus obtain $h(G)\leq\Omega(\exp(G))$ as in the proof of Theorem 1.
\phantom{.}\hfill$\Box$

%%%%%%%%%%%%%%%%%% Bibliography %%%%%%%%%%%%%%%%%%%%%%%%%%%%%%%%%%%%%

\bibliographystyle{plain}
\bibliography{Fitting}

\end{document}